\documentclass[11pt]{amsart}

\usepackage{amsmath,amsthm,amssymb,latexsym,graphicx,verbatim,enumerate,
ifpdf,mdwlist,xspace}
\usepackage{mathabx}
\ifpdf
\usepackage{hyperref}
\else
\usepackage[hypertex]{hyperref}
\fi
\usepackage{color}
\usepackage{amssymb, amsmath, amsthm}
\usepackage{graphicx}
\usepackage{cancel}
\usepackage{float}

\newtheorem{question}{Question}
\newtheorem{thm}{Theorem}[section]
\newtheorem{definition}[thm]{Definition}

\newtheorem{remark}[thm]{Remark}

\newtheorem{lemma}[thm]{Lemma}

\newtheorem{prop}[thm]{Proposition}
\newtheorem{corollary}[thm]{Corollary}

\renewcommand{\phi}{\varphi}
\renewcommand{\setminus}{\smallsetminus}

\title[Minimal equivalence relations]{Minimal Equivalence Relations\\
in Hyperarithmetical and Analytical Hierarchies}

\author[N.~Bazhenov]{Nikolay Bazhenov}
\address{Sobolev Institute of Mathematics, 4 Acad. Koptyug Ave., Novosibirsk, 630090 Russia \\
Department of Mathematics and Mechanics, Novosibirsk State University, 2 Pirogova St., Novosibirsk, 630090 Russia}
\email{\href{mailto:bazhenov@math.nsc.ru}{bazhenov@math.nsc.ru}}

\author[M.~Mustafa]{Manat Mustafa}
\address{Department of Mathematics, School of Sciences and Humanities, Nazarbayev University,
53 Qabanbay Batyr Ave., Nur-Sultan, 010000 Kazakhstan}
\email{\href{manat.mustafa@nu.edu.kz}{manat.mustafa@nu.edu.kz}}

\author[L.~San Mauro]{Luca San Mauro}
\address{Institute of Discrete Mathematics and Geometry, Vienna University of Technology, Vienna, Austria}
\email{luca.san.mauro@tuwien.ac.at}

\author[M.~Yamaleev]{Mars Yamaleev}
\address{N.I.
Lobachevskii Institute of Mathematics and Mechanics, Kazan (Volga
Region) Federal University, 18 Kremlyovskaya str., Kazan, 420008
Russia}
\email{\href{mailto:mars.yamaleev@kpfu.ru}{mars.yamaleev@kpfu.ru}}

\thanks{The work was supported by Nazarbayev University Faculty Development Competitive Research Grants N090118FD5342. Part of the research contained in this paper was carried out while Bazhenov, San Mauro, and Yamaleev were visiting the Department of Mathematics of Nazarbayev University, Nur-Sultan. The authors wish to thank Nazarbayev University for its hospitality. Bazhenov was supported by the
grant of the President of the Russian Federation
(No.~MK-1214.2019.1).  San Mauro was supported by the Austrian
Science Fund FWF, project~M~2461. Yamaleev was supported by the
Russian Science Foundation, project No.~18-11-00028.}



\begin{document}

\maketitle

\begin{abstract}
A standard tool for classifying the complexity of equivalence relations on $\omega$ is provided by computable reducibility. This reducibility gives rise to a rich degree structure. The paper studies equivalence relations, which induce minimal degrees with respect to computable reducibility. Let $\Gamma$ be one of the following classes: $\Sigma^0_{\alpha}$, $\Pi^0_{\alpha}$, $\Sigma^1_n$, or $\Pi^1_n$, where $\alpha \geq 2$ is a computable ordinal and $n$ is a non-zero natural number. We prove that there are infinitely many pairwise incomparable minimal equivalence relations that are properly in $\Gamma$.
\end{abstract}

\section{Introduction}

The paper studies recursion-theoretic complexity of equivalence relations on the domain $\omega$. Our main working tool is \emph{computable reducibility}.

\begin{definition}
    Let $R$ and $S$ be equivalence relations on the domain $\omega$. The relation $R$ is \emph{computably reducible} to $S$ (denoted by $R \leq_c S$) if there is a computable function $f(x)$ such that for all $x,y\in\omega$, the following holds: $(xRy)\ \Leftrightarrow\ (f(x)Sf(y))$.
\end{definition}

We write $R\equiv_c S$ if $R\leq_c S$ and $S\leq_c R$. Throughout the paper, we assume that every considered equivalence relation has domain $\omega$.

The systematic study of $c$-degrees, i.e. degrees induced by computable reducibility, was initiated by Ershov~\cite{Ershov-71,Ershov-Book}. His approach is motivated by the theory of numberings, specifically by its category-the\-o\-re\-tic facets. In 1980s, the research of $c$-degrees was concentrated on classifying the complexity of computably enumerable equivalence relations (or \emph{ceers} for short): in particular, the  provable equivalence in formal systems was in the spotlight, see, e.g., \cite{Ber-81,BS-83}. Note that the acronym \emph{ceer} was introduced in the paper~\cite{GG-01}. Andrews and Sorbi~\cite{AS-ta} provided a deep analysis of algebraic properties for the $c$-deg\-rees of ceers. For a detailed exposition of the state-of-the-art results on ceers, the reader is referred to, e.g., \cite{AS-ta,ABS-survey,AS-18}.

The recent works~\cite{NgYu,BMSSY,BK-ta} started systematic investigations of $c$-degrees for $\Delta^0_2$ equivalence relations. We note that computable reducibility has been also studied for higher levels of the hyperarithmetical hierarchy, \emph{but} these studies were largely focused on \emph{complete} equivalence relations.

Let $\Gamma$ be a complexity class (e.g., $\Pi^0_1$, $\Sigma^0_n$,
or $\Sigma^1_1$). An equivalence relation $R$ is called
\emph{$\Gamma$-complete} (under computable reducibility) if $R\in
\Gamma$ and every equivalence relation $E\in \Gamma$ satisfies
$E\leq_c R$. Known examples of $\Gamma$-complete equivalence
relations include:
\begin{itemize}
    \item The relation of provable equivalence in Peano arithmetic is $\Sigma^0_1$-complete \cite{BS-83}.

    \item $1$-equivalence and $m$-equivalence on indices of c.e. sets are both $\Sigma^0_3$-complete \cite{FFN-12}.

    \item The relation of computable isomorphism on (computable indices for) the class of computable Boolean algebras is $\Sigma^0_3$-complete \cite{FFN-12}.

    \item For every natural number $n$, $1$-equivalence on indices of $\emptyset^{(n+1)}$-c.e. sets is $\Sigma^0_{n+4}$-complete \cite{IMNN-14}.

    \item For every computable successor ordinal $\alpha$, the relation of $\Delta^0_{\alpha}$ isomorphism on the class of computable distributive lattices is $\Sigma^0_{\alpha+2}$-complete \cite{BMY-19}.

    \item The isomorphism relation on the class of computable linear orders is $\Sigma^1_1$-complete \cite{FFH-12}
\end{itemize}
For further results on $\Gamma$-complete equivalence relations, we refer the reader to, e.g., \cite{IMNN-14}.

The goal of this paper is to investigate hyperarithmetical equivalence relations, which are \emph{far from} being $\Gamma$-complete. Note the following simple fact: if an equivalence relation $R$ has infinitely many classes, then for every computable equivalence relation $F$ having only finitely many classes, we have $F\leq_c R$. This observation suggests the following natural notion of minimality.

For a non-zero natural number $n$, by $\mathrm{Id}_n$ we denote the following equivalence relation:
\[
    (x,y) \in \mathrm{Id}_n \ \Leftrightarrow\ n \text{ divides } (x-y).
\]
Clearly, if a computable equivalence relation $F$ has precisely $n$ classes, then $F$ is $\equiv_c$-equiva\-l\-ent to $\mathrm{Id}_n$.

\begin{definition}[essentially formulated in Theorem~3.3 of \cite{AS-ta}]
    We say that an equivalence relation $R$ is \emph{minimal} if $R$ has infinitely many equivalence classes and for any equivalence relation $E$, the following holds:
    \[
        E\leq_c R\ \Rightarrow\ (E\equiv_c R) \vee (\exists n)(E\equiv_c \mathrm{Id}_n).
    \]
\end{definition}

It is not hard to see that the identity relation $\mathrm{Id}$ is minimal. Furthermore, Andrews and Sorbi (Theorem~3.3 of \cite{AS-ta}) proved that there are minimal ceers $E_i$, $i\in\omega$, such that they are pairwise $\leq_c$-in\-com\-pa\-rable and $\mathrm{Id} \nleq_c E_i$ for every $i$.

For a complexity class $\Gamma$, by $\breve{\Gamma}$ we denote the \emph{dual class} of $\Gamma$. For example, if $\Gamma = \Sigma^0_{\alpha}$, then $\breve{\Gamma} = \Pi^0_{\alpha}$. If $\Gamma = \Pi^1_n$, then $\breve{\Gamma} = \Sigma^1_n$. We say that an equivalence relation $R$ is a \emph{proper $\Gamma$ relation} if $R$ belongs to $\Gamma \setminus \breve{\Gamma}$.

The structure of the paper is as follows. Section~\ref{sect:main-res} contains a general sufficient condition for the existence of minimal equivalence relations (Theorem~\ref{theo:main}). Section~\ref{sect:consequences} discusses the consequences of Theorem~\ref{theo:main}.   For every computable ordinal $\alpha \geq 2$, we show that there are infinitely many pairwise $\leq_c$-in\-com\-pa\-rable, minimal, proper $\Sigma^0_{\alpha}$ equivalence relations. Similar results are obtained for the classes $\Pi^0_{\alpha}$, $\Sigma^1_n$, and $\Pi^1_n$, where $1 \leq n < \omega$.


\section{Existence of Minimal Equivalence Relations} \label{sect:main-res}

This section proves the following sufficient condition for the
existence of minimal equivalence relations (by $\Sigma^0_1(X)$ we
denote the sets which are $\Sigma^0_1$ with oracle $X$):

\begin{thm} \label{theo:main}
    Let $X$ be an oracle such that $X\geq_T \emptyset'$.
    \begin{itemize}
        \item[(a)] There are minimal equivalence relations $E_i$, $i\in\omega$, such that $E_i$ are pairwise $\leq_c$-incomparable, and $E_i \in \Sigma^0_1(X)\setminus \Pi^0_1(X)$ for every $i$.

        \item[(b)] There are minimal equivalence relations $F_i$, $i\in\omega$, such that $F_i$ are pairwise $\leq_c$-incomparable, and $F_i \in \Pi^0_1(X)\setminus \Sigma^0_1(X)$ for every $i$.
    \end{itemize}
    Furthermore, for every $i\in\omega$, every $E_i$-class and every $F_i$-class are computably enumerable.
\end{thm}

Before proving Theorem~\ref{theo:main}, we give two useful facts about minimal equivalence relations. Recall that a ceer $R$ is called \emph{dark} if $R$ is incomparable with $\textrm{Id}$ under computable reducibility (Definition~3.1 of \cite{AS-ta}).

\begin{prop}[Andrews and Sorbi~{\cite{AS-ta}}] \label{prop:min-ceer}
    Let $R$ be a dark ceer. Then the following conditions are equivalent:
    \begin{enumerate}
        \item $R$ is minimal.

        \item For any c.e. set $W$, if $W$ intersects infinitely many $R$-classes, then $W$ intersects all $R$-clas\-ses.
    \end{enumerate}
\end{prop}
\begin{proof}
    This fact follows from Lemmas~3.4 and~3.5 of~\cite{AS-ta}, but for the sake of completeness, we outline the proof of the fact.

    $(1)\Rightarrow(2)$. Suppose that there is a c.e. set $W$ such that $W$ intersects infinitely many, but not all $R$-classes. Fix a computable injective function $g(x)$ with $range(g)=W$, and define a ceer $S$ as follows:
    $ (x S y)\ \Leftrightarrow\ (g(x) R g(y))$.
    Clearly, $S\leq_c R$, and $S$ has infinitely many classes. In order to prove that $R$ is not minimal, it is sufficient to show that $S\not\equiv_c R$.

    Towards a contradiction, assume that $R\leq_c S$ via a computable function $f$. Choose an element $a\in\omega$ such that $[a]_R \cap W =\emptyset$, and consider a sequence of numbers defined as follows: $a_0 := a$ and $a_{n+1} := g(f(a_n))$. We claim that for any $i<j$, the elements $a_i$ and $a_j$ are not $R$-equivalent. Indeed, if $(a_i R a_j)$, then we have the following sequence of implications:
    \begin{multline*}
        (g(f(a_{i-1})) R g(f(a_{j-1})))\ \Rightarrow\ (f(a_{i-1}) S f(a_{j-1}))\ \Rightarrow\ (a_{i-1} R a_{j-1})\  \Rightarrow \\ \ (a_{i-2} R a_{j-2}) \  \Rightarrow\ \dots
        \Rightarrow\ (a_0 R a_{j-i}),
    \end{multline*}
    where $a_{j-i} = g(f(a_{j-i-1})) \in W$. Thus, $W$ intersects with the class $[a]_R$, which contradicts the choice of $a$. Hence, now we know that the elements $a_i$, $i\in\omega$, are pairwise not $R$-equivalent.

    This shows that the function $h(x):=a_x$ provides a reduction $\mathrm{Id} \leq_c R$, which contradicts the darkness of $R$. Therefore, we obtain that $S <_c R$, and $R$ is not minimal.

    $(2)\Rightarrow(1)$. Suppose that $R$ satisfies the second condition. Consider an arbitrary ceer $E$ with infinitely many classes such that $E\leq_c R$ via a function $f$. In order to finish the proof, it is sufficient to show that $R\leq_c E$.

    The c.e. set $range(f)$ intersects infinitely many $R$-classes, and hence, $range(f)$ intersects all $R$-classes. Therefore, the desired reduction $g$ from $R$ into $E$ can be defined as follows: for $x\in\omega$, choose $g(x)$ as a number $y_x$ such that $f(y_x)$ is the first (under a fixed enumeration of the ceer $R$) element with $f(y_x)\in [x]_R$. Clearly, we have: $(xRx')$ iff $(f(y_x) R f(y_{x'}))$ iff $(g(x) E g(x'))$. Proposition~\ref{prop:min-ceer} is proved.
\end{proof}

Proposition~\ref{prop:min-ceer} implies the following fact about equivalence relations, which are \emph{not} necessarily ceers:

\begin{prop}\label{prop:min-arbitrary}
    Let $E$ be a dark minimal ceer, and let $R$ be an arbitrary equivalence relation such that $R$ has infinitely many classes and $R\supseteq E$. Then $R$ is minimal.
\end{prop}
\begin{proof}
    Suppose that $S$ is an equivalence relation, and $f$ is a computable reduction from $S$ into $R$. Then precisely one of the following two cases holds:

    \emph{Case~1.} Assume that the set $range(f)$ intersects only finitely many $E$-classes. We emphasize that here we consider the classes of the ceer $E$, but not $R$-classes. Evidently, in this case $S$ also has finitely many classes.

    Then in a non-uniform way, we choose representatives $a_0,a_1,\dots,a_m$ of all $E$-classes which intersect $range(f)$. Since $E$ is a ceer, the function $h\colon x\mapsto a_i$, where $f(x) \in [a_i]_E$, is computable. Clearly, the condition $(xSx')$ is equivalent to $(h(x)Rh(x'))$. Since the set $range(h)$ is finite, we deduce that the relation $S$ is computable, and $S\equiv_c \mathrm{Id}_k$ for some $k\in\omega$.

    \emph{Case~2.} Assume that $range(f)$ intersects infinitely many $E$-classes. Then by Proposition~\ref{prop:min-ceer}, $range(f)$ intersects \emph{all} $E$-classes.

    We define a computable function $g$ as follows: for an element $x\in\omega$, choose $g(x)$ as a number $z_x$ such that the value $f(z_x)$ is the first (under a fixed enumeration of $E$) number with $f(z_x) \in [x]_E$. We claim that the function $g$ reduces $R$ to $S$. Indeed, since $E\subseteq R$, for arbitrary $x$ and $x'$, we have:
    \[
        (x R x')\ \Leftrightarrow\ (f(z_x) R f(z_{x'}))\ \Leftrightarrow\ (z_x S z_{x'})\ \Leftrightarrow\ (g(x) S g(x')).
    \]
    Therefore, we showed that $S\equiv_c R$. Hence, $R$ satisifies the definition of minimality. Proposition~\ref{prop:min-arbitrary} is proved.
\end{proof}

Now we are ready to obtain the main result of the section. By $\leq_{\omega}$ we denote the standard ordering of natural numbers.

\begin{proof}[Proof of Theorem~\ref{theo:main}]
    Recall that Andrews and Sorbi (Theorem~3.3 of \cite{AS-ta}) proved that there are infinitely many pairwise $\leq_c$-incomparable, dark minimal ceers.

    We choose just one such ceer $R$, and we find the sequence $(a_i)_{i\in\omega}$ containing the $\leq_{\omega}$-least representatives from all $R$-classes. More formally, this means that any number $x\in\omega$ is $R$-equivalent to some $a_i$, and for any $y<_{\omega} a_i$, $y$ is not $R$-equivalent to $a_i$.  Since $R$ is a ceer, it is clear that the sequence $(a_i)_{i\in\omega}$ is $\mathbf{0}'$-com\-pu\-table.

    The following auxiliary result can be obtained via an easy relativization of Exercise~2.2.(a) from Chapter~VII in~\cite{Soare}, so the proof of this result is omitted.

    \begin{lemma}\label{lem:aux}
        There is a uniform sequence of $X$-c.e. sets $(B_i)_{i\in\omega}$ such that for all $i\neq j$, we have $X \leq_T B_i \nleq_T B_j \oplus X$.
    \end{lemma}

    We prove item~(a) of the theorem. For an index $k\in \omega$, define an equivalence relation $E_k$ as follows: $E_k$ is the $\subseteq$-least equivalence relation such that
    \[
        E_k \supseteq R \cup \{ (a_{2j}, a_{2j+1}) \,\colon j\in B_k \}.
    \]
    Since $X\geq_T \emptyset'$ and the set $B_k$ is c.e. in $X$, it is clear that $E_k \in \Sigma^0_1(X)$. Moreover, it is not difficult to show that $E_k \leq_T B_k \oplus \emptyset' \leq_T B_k \oplus X \equiv_T B_k \leq_T E_k \oplus \emptyset'$. Note that any $E_k$-class is equal either to an $R$-class, or to a union of two $R$-classes. Thus, every $E_k$-class is a c.e. set.

    Since $B_k \nleq_T X$ and $B_k\leq_T E_k \oplus \emptyset' \leq_T E_k \oplus X$, we deduce that $E_k \nleq_T X$ and $E_k\not\in \Pi^0_1(X)$. Furthermore, $E_k\supseteq R$ and $E_k$ has infinitely many classes, hence, by Proposition~\ref{prop:min-arbitrary}, $E_k$ is minimal.

    Assume that $E_k \leq_c E_l$ for some $k\neq l$. Then we have $B_k \leq_T E_k \oplus \emptyset' \leq_T E_l \oplus \emptyset' \leq_T B_l\oplus \emptyset' \leq_T B_l \oplus X$, which contradicts the choice of the sequence $(B_i)_{i\in\omega}$. Therefore, the sequence of equivalence relations $(E_k)_{k\in\omega}$ has all desired properties.

    The proof of item~(b) of the theorem is essentially the same as that of the item~(a), modulo the following key modification: the relation $F_k$ is the $\subseteq$-least such that $F_k \supseteq R \cup \{ (a_{2j}, a_{2j+1}) \,\colon j\not\in B_k \}$.
    This concludes the proof of Theorem~\ref{theo:main}.
\end{proof}


\section{Consequences of the Main Result} \label{sect:consequences}

Theorem~\ref{theo:main} immediately implies the following fact:

\begin{corollary}\label{conseq:01}
    Let $\alpha \geq 2$ be a computable ordinal. There are infinitely many pairwise $\leq_c$-in\-com\-pa\-ra\-ble, minimal, proper $\Sigma^0_{\alpha}$ equivalence relations. A similar result holds for the class $\Pi^0_{\alpha}$.
\end{corollary}
\begin{proof}
    Choose the oracle
    \[
        X := \begin{cases}
            \emptyset^{(\alpha - 1)}, & \text{if } \alpha <\omega,\\
            \emptyset^{(\alpha)}, & \text{if } \alpha \geq \omega,
        \end{cases}
    \]
    in Theorem~\ref{theo:main}.
\end{proof}

Note that Corollary~\ref{conseq:01} cannot be extended to the $\Pi^0_1$-case: it is not hard to show that for any $\Pi^0_1$ equivalence relation $E$ with infinitely many classes, we have $\mathrm{Id} \leq_c E$ (see, e.g., Proposition~3.1 of \cite{BMSSY}).

The ideas of the proof of Theorem~\ref{theo:main} also help us to deal with the levels of the analytical hierarchy:

\begin{prop}\label{prop:analytical}
    Let $n$ be a non-zero natural number. There are infinitely many pairwise $\leq_c$-in\-com\-pa\-rable, minimal, proper $\Pi^1_n$ equivalence relations. A similar result holds for the class $\Sigma^1_n$.
\end{prop}
\begin{proof}
    As in the proof of Theorem~\ref{theo:main}, we fix a dark minimal ceer $R$ and the sequence $(a_i)_{i\in\omega}$ containing the $\leq_{\omega}$-least representatives of all $R$-classes.

    Let $B$ be an $m$-complete $\Pi^1_n$ set. Choose an arbitrary sequence $(C_k)_{k\in\omega}$ of hyperarithmetical sets such that $C_k$ are pairwise Turing incomparable and $C_k\geq_T \emptyset^{(2)}$ for all $k$. Such a sequence can be obtained, e.g., by applying Lemma~\ref{lem:aux} to the oracle $X = \emptyset^{(2)}$. For an index $k\in\omega$, the relation $E_k$ is the $\subseteq$-least equivalence relation such that
    \[
        E_k \supseteq R \cup \{ (a_{2i}, a_{2j}) \,\colon i,j \in B\} \cup \{ (a_{2i+1},a_{2j+1})\,\colon i,j \in C_k\}.
    \]
    Since the set $\omega\setminus B$ is infinite, $E_k$ has infinitely many equivalence classes. Thus, by Proposition~\ref{prop:min-arbitrary}, $E_k$ is minimal.

    Define a $\mathbf{0}'$-computable total function $g(x)$ as follows: for a number $x$, $g(x)$ is equal to the index $i$ such that $a_i\in [x]_R$. It is not hard to show that the condition $(x E_k y)$ is true if and only if at least one of the following conditions holds:
    \begin{itemize}
        \item[(a)] $g(x) = g(y)$;

        \item[(b)] both values $g(x)$ and $g(y)$ are odd, $[g(x)/2] \in C_k$, and $[g(y)/2]\in C_k$;

        \item[(c)] both values $g(x)$ and $g(y)$ are even, $[g(x)/2] \in B$, and $[g(y)/2] \in B$.
    \end{itemize}
    The last condition can be re-written in the following form:
    \[
        \exists u \exists v[ (g(x) = 2u) \,\&\, (g(y) = 2v) \,\&\, (u\in B) \,\&\, (v\in B) ].
    \]
    Therefore, a standard application of the Tarski--Kuratowski algorithm shows that the relation $E_k$ is $\Pi^1_n$.

    Note that $B\leq_T E_k \oplus \emptyset'$. Towards a contradiction, assume that $E_k$ is a $\Delta^1_n$ relation. Then the set $E_k\oplus \emptyset'$ is $\Delta^1_n$, and $B$ is $\Delta^1_1$ relative to $E_k \oplus \emptyset'$. By the result of Shoenfield (see, e.g., Proposition~5.2 in Chapter~II of~\cite{Sacks}), we deduce that $B$ is a $\Delta^1_n$ set, which gives a contradiction. Therefore, $E_k$ is a proper $\Pi^1_n$ relation.

    In order to prove that $E_k$, $k\in\omega$, are pairwise $\leq_c$-incomparable, we employ the following easy observation: Let $S$ and $T$ be arbitrary equivalence relations. If a computable function $f$ provides a reduction $S\leq_c T$, then for every element $x_0\in\omega$, we have $f\colon [x_0]_S \leq_m [f(x_0)]_T$.

    Without loss of generality, we may assume that $0 \in B\cap C_k$ for all $k$. Then it is not hard to show that $E_k$ has only two equivalence classes which are not c.e.~--- the classes of $a_0$ and $a_1$. Indeed, the class $[a_0]_{E_k}$ is not even hyperarithmetical. Moreover, $\emptyset^{(2)} \leq_T C_k \leq_T [a_1]_{E_k} \oplus \emptyset'$ and $[a_1]_{E_k} \leq_T C_k \oplus \emptyset' \equiv_T C_k$.

    Assume that a computable function $f$ gives a reduction $E_k \leq_c E_l$ for some $k\neq l$. Then by employing the observation above, we consider the $m$-degrees of the equivalence classes, and we deduce that $f(a_0) \in [a_0]_{E_l}$ and $f(a_1) \in [a_1]_{E_l}$. Hence, we have $f\colon [a_1]_{E_k} \leq_m [a_1]_{E_l}$. Thus, $C_k \leq_T [a_1]_{E_k} \oplus \emptyset' \leq_T [a_1]_{E_l} \oplus \emptyset' \leq_T C_l \oplus \emptyset' \equiv_T C_l$, which contradicts the choice of the sequence $(C_i)_{i\in\omega}$. Therefore, the relations $E_k$, $k\in\omega$, are pairwise $\leq_c$-incomparable.

    The proof for $\Sigma^1_n$ equivalence relations is essentially the same, modulo the following modification: one needs to choose $B$ as an $m$-complete $\Sigma^1_n$ set. Proposition~\ref{prop:analytical} is proved.
\end{proof}

Note that the equivalence relations $E_k$, $k\in\omega$, of Proposition~\ref{prop:analytical} are more intricate than those of Theorem~\ref{theo:main}: now each $E_k$ has precisely two non-c.e. classes.

{\bf Remark.} The desired $c$-degrees from Corollary~\ref{conseq:01}
and Proposition~\ref{prop:analytical} are proper for a given level
and also dark. This extends some results about proper and dark
$c$-degrees from~\cite{BMSSY}.

\bigskip

\end{document}